\documentclass[12pt,reqno]{article}

\usepackage[utf8]{inputenc}
\usepackage{amsmath}
\usepackage{amssymb}
\usepackage{amsthm}
\usepackage{verbatim}
\usepackage[affil-it]{authblk}
\usepackage{mathtools}
\usepackage{hyperref}

\title{The Repeated Divisor Function and Possible Correlation with Highly Composite Numbers}
\author{Sayak Chakrabarty%
  \thanks{Electronic address: \texttt{sayakc@cmi.ac.in}; }}
\affil{Chennai Mathematical Institute\\ Mathematics and Computer Science}

\author{Arghya Datta%
  \thanks{Electronic address: \texttt{arghyad@cmi.ac.in}}}
\affil{Chennai Mathematical Institute\\ Mathematics and Theoretical Physics}

\begin{document}
\maketitle
\hspace{40 mm} DEDICATED TO R.BALASUBRAMANIAN
\theoremstyle{plain}
\newtheorem{theorem}{Theorem}
\newtheorem{corollary}[theorem]{Corollary}
\newtheorem{lemma}[theorem]{Lemma}
\newtheorem{proposition}[theorem]{Proposition}
\theoremstyle{definition}
\newtheorem{definition}[theorem]{Definition}
\newtheorem{example}[theorem]{Example}
\newtheorem{conjecture}{Conjecture}
\theoremstyle{remark}
\newtheorem{remark}[theorem]{Remark}

\begin{abstract}
Let $n$ be a positive integer and $d(n)$ be the number of positive divisors of $n$. Let us call $d(n)$ the divisor function. Of course, $d(n) \leq n$. $d(n) = 1$ if and only if $n = 1$. For $n > 2$ we have $d(n) \geq 2$ and in this paper we try to find the smallest $k$ such that $d(d(\cdots(n)\cdots)) = 2$ where the divisor function is applied $k$ times, and we define $k$ to be the period for the number $n$. At the end of the paper we make a conjecture based on some observations.
\end{abstract}

\section{Introduction}
We found this problem in a paper by Florentin Smarandache, and for further information see [3]. This is the 18th unsolved problem in his paper.\\
We start with some trivial observations. $d(d(\cdots(n)\cdots)) = 2$ implies $d^{k-1}(n) = p$ where $p$ is a prime. If $p=2$ then the chain continues infinitely long without any significance.\\
Otherwise suppose $p$ is odd, $p=2\alpha + 1$. We know that only perfect squares have odd number of factors and since that odd number $2\alpha + 1$ is prime the only choice for the perfect square is $q^{2\alpha}$ where $q$ is a prime. Now  $q$ can be arbitrarily large.\\
Going one more step backwards, we see that a number having $q^{2\alpha}$ divisors will be of the form  $M=\prod_{i=1}^{m} p_{i}^{q^{a_i}-1}$ for some $1\leq m \leq 2\alpha$. Here $p_{i}$ are distinct primes and $\sum_{i}a_{i}=2\alpha$. Though we can fix $\alpha$, $M$ can be arbitrarily large since $q$ can be arbitrarily large.\\

\section{The possible values of period for a given integer }
From introduction we clearly observe that any given integer $n$ can be arbitrarily large while it's period, $k=3$ remains fixed and we get $d^{3}(n) = 2$ at the end. But computer programming reveals that if we plot $k$ with respect with $n$, the frequency with which $k=3$ or $k=4$ occurs is far above than any other frequency for at least  up to numbers like $5000000$. $k=5$ first occurs at 60 and $k=6$ first occurs at 5040. $k=7$ first occurs when $n = 2^{6}$ x $3^{4}$ x $5^{2}$ x $7{2}$ x $11$ x $13$ x $17$ x $19$ which is more than 10 digit number. We observe that $k$ increases very slowly compared to $n$. But what is interesting is that $k=3$ or $k=4$ occurs with same frequency almost in every sufficiently large interval. $k=1$ also sometimes occurs  due to the distribution of primes and the presence of twin primes. \\
But we can clearly see here that $k$ attains every integer $m \in \mathbb{N}$. Observe that \\
given $n = \prod_{i = 1}^{m} p_{i}^{a_{i}}$ and $k=r$ we just construct $n_{1}$ such that $d(n_{1}) = n$, then for $n_{1}$ we have $k= r+1$.
Just put $n_{1} = \prod_{i}^{m} q_{i}^{p_{i}^{a_{i}} - 1}$ where $q_{j}$ is the $j^{th}$ prime starting from 2. So $k$ is unbounded.

\section{The least integer for a given period }
After the previous section, here we give a theorem  which will allow us to give the smallest $n_{1}$ for which $k = r+1$ from a given integer $n$ with its known period $r$. Since we know that $60$ is  the smallest number where $k=5$ the first time, by induction we can consequently find the $n_{1}'s$ for which $k= 6,7,8 \ldots$.\\ 
Look at the following image on the next page to get an idea of the variation of $k$ with respect to $n$ when $n$ is taken in the range $(0,350)$. We plot the $n$ along the $x$ axis and the corresponding $k$ along the $y$ axis.

\newpage

\begin{theorem} 
Given an integer $n$ with it's prime decomposition, $n=p_{1}^{a_{1}}.p_{2}^{a_{2}}\cdots p_{m}^{a_{m}}$. Suppose further that period for $n$ is $k$ . If  $L$  the smallest integer whose period is $(k+1)$ then prime factorization of $L$  is given by 
$L= \bigg(2^{p_{m}- 1}.3^{p_{m}- 1} \cdots p_{a_{m}}^{p_{m}- 1}\bigg) \bigg(p_{a_{m}+ 1}^{p_{m-1}- 1}.p_{a_{m}+ 2}^{p_{m-1}- 1} \cdots p_{a_{m}+ a_{m-1}}^{p_{m-1}- 1}\bigg) \bigg(p_{a_{m}+a_{m-1}+ 1}^{p_{m-2}- 1} \cdots \bigg) $.

\end{theorem}

\begin{proof}
First of all, assume $L=p_{1}^{b_{1}}.p_{2}^{b_{2}}\cdots p_{\alpha}^{a_{\alpha}}$. Here, we again mention that $(p_i)_{i\geq 1}$  is an enumeration of primes in increasing order , i.e. $p_1=2,p_2=3\ldots$ and so on ! So in particular those $a_i$ or $b_j$'s could be $0$ as well, in case their corresponding prime is absent in the decomposition.

\textbf{Case 1: $a_{i} = 1$ for some $i$}\\
In order to construct the minimum $L$, we need to make sure that the largest prime should be put as the index on the smallest possible prime. So if $a_{i} = 1$ for some $i$, clearly it goes to power of single prime because if $a_{m} = 1$ without loss of generality, then $b_{1} = p-1$ because otherwise, $L$ will not me minimal.\\

\textbf{Case 2: $a_{i} \geq 2$ for some $i$}\\
Here we say that for a generic term in  prime decomposition say $p_{j}^{a_{j}}$, it can be distributed like $2^{p_{j}^{a_{j}} - 1}$ or $2^{p_{j} - 1}.3^{p_{j} - 1}\cdots p_{a_{j}}^{p_{j} - 1}$ two ways.We will prove that to achieve the minimal $L$, the second choice is better. Similarly, we can argue $3^{p_{j}^{a_{j}} - 1} > 3^{p_{j} - 1}\cdots p_{a_{j+1}}^{p_{j} - 1}$. This will lead to the conclusion that each generic coupe, say without loss of generality $p_{m}^{a_{m}}$ will give $(2^{p_{m} - 1}.3^{p_{m} - 1}\cdots p_{a_{m}}^{p_{m} - 1})$ contribution in the prime factorization of $L$.\\
Now, We will use induction on $a_{j}$.\\
For $a_{j} = 2$, without loss of generality let $j=m$. If $j = k(<m)$ then instead of 2, our decomposition will start with $p_{a_{m}+ a_{m-1} + \cdots+ a_{k-1}+ 1}$ and argument for that will be similar.\\
If $a_{j} = 2$ we have to show:
\begin{eqnarray}
2^{p_{m}^{2}- 1} > 2^{p_{m} -1}.3^{p_{m} -1}\\
\implies 2^{p_{m}} > 3
\end{eqnarray}
\textbf{Induction Step:} Assuming $a_{m} = k$ we will prove for $a_{m} = k+1$\\
$2^{p_{m}^{k+1} - 1} > (2^{p_{m} - 1}. 3^{p_{m} - 1} \cdots p_{a_{m} - 1}^{p_{m} - 1})(p_{a_{m}}^{p_{m} - 1})$\\

Now $(2^{p_{m} - 1}.3^{p_{m} - 1} \cdots p_{a_{m} - 1}^{p_{m} - 1}) < 2^{p_{m}^{k}- 1}$ by the hypothesis.\\
So it is enough to check if \\
\begin{eqnarray}
2^{p_{m}^{k+1} - 1} > 2^{p_{m}^{k} - 1} . p_{a_{m}}^{p_{m}- 1}\\
\implies 2^{p_{m}^{k+1} - p_{m}^{k}} > p_{a_{m}}^{p_{m} - 1}\\
\implies 2^{p_{m}^{k}} > p_{a_{m}}
\end{eqnarray}
Now it is clearly true that $p_{n} \leq 2^{n}$ and so enough to show\\
$2^{p_{m}^{k}} \geq 2^{k+1}$. But clearly $p_{m}^{k} > k+1$, and so we are done.\\
\end{proof}
\textbf{Example 1.}
If we put $n = 5040$ then we get $L = 2^{6}.3^{4}.5^{2}.7^{2}.11.13.17.19$ which is a 13 digit number. Observe how we use the theorem.\\
$5040 = 2^{4}.3^{2}.5.7$. So according to our theorem since 5 and 7 have index 1, they will correspond to a single prime number each. We have to construct $L$ such that $d(L) = 5040$. So the prime factorization of $L$ will begin with $2^{6}.3^{4}$ for sure. Now to get $3^{2}$ as a factor of $d(L)$ we need to distribute it in such a way that our obtained $L$ is minimum.\\ 
So we have $ L = 2^{6}.3^{4}.5^{2}.7^{2} \cdots$ and by similar reasoning we finish the construction of $L$ as $L = 2^{6}.3^{4}.5^{2}.7^{2}.11.13.17.19$.\\
It is noticeable that the theorem shows it is always better to distribute the indexes over as many primes as possible to minimize the outcome.

\section{An estimated behavior of period for a given integer}
\begin{theorem}
Given any integer $n$ it's period has size $O(\log n)$ 
\end{theorem}
\begin{proof}
Constructing $n_{1}$ from $n$ according to Theorem 1, we see that if $n$ has prime decomposition of the form $p_{1}^{a_{1}}.p_{2}^{a_{2}}\cdots p_{m}^{a_{m}}$ then the same for $n_{1}$ will be as follows:\\
$n_{1} = \bigg(2^{p_{m}- 1} 3^{p_{m}- 1} \cdots p_{a_{m}}^{p_{m}- 1}\bigg) \bigg(p_{a_{m}+ 1}^{p_{m-1}- 1} p_{a_{m}+ 2}^{p_{m-1}- 1} \cdots p_{a_{m}+ a_{m-1}}^{p_{m-1}- 1}\bigg) \bigg(p_{a_{m}+a_{m-1}+ 1}^{p_{m-2}- 1} \cdots \bigg)\cdots $\\

So $\log n = \sum_{i= 1}^{m} a_{i} \log p_{i}$ and also\\
$\log n_{1} = (p_{m}-1)\log [2.3 \cdots p_{a_{m}}] + (p_{m}-1) \log (p_{a_{m}+1} \cdots p_{a_{m}+a_{m-1}}) + \cdots$\\

Now we will use a well known fact that product of first $n$ primes is asymptotically $e^{n\log n}$. Using this above result changes the above equation as follows:\\

$\log n_{1} = (p_{m}-1)a_{m}\log a_{m} + (p_{m-1}-1)\bigg[(a_{m} + a_{m-1})\log (a_{m}+ a_{m-1}) - a_{m}\log a_{m}\bigg] + (p_{m-2}-1)\bigg[(a_{m} + a_{m-1}+ a_{m-2})\log (a_{m}+ a_{m-1} + a_{m-2}) - (a_{m}+ a_{m-1}\log (a_{m} + a_{m-1})\bigg] + \cdots$\\

Now to compare $\log n_{1}$ to $\log n$ we will investigate the increment for each $a_{i}$'s. We have to begin with the coefficient for $a_{m}$ in $\log n_{1}$ .\\
Observe that $(p_{m}-1)a_{m}\log a_{m}$ serves as the main term since except this term, others involve decreasing functions which can be arbitrarily small but all these terms are clearly non-negative.\\
This follows because \\
$a_{i} \geq 2$ and $\log (n+2) - \log n = \log (1+ \dfrac{2}{n}) \to 0$ as $n \to \infty $.\\
The assumption that $a_{i} \geq 2$ will be justified shortly.\\
So the main contribution is due to $(p_{m}-1)a_{m}\log a_{m}$. And similarly main term related to increment for the co-efficient of $a_{m-1}$ will be $(p_{m-1}-1)(a_{m} + a_{m-1})\log (a_{m}+ a_{m-1})$ which is greater than $(p_{m}-1)a_{m}\log a_{m}$. An interesting thing to observe is that the above inequality cannot be considerably made better since $a_{m}$ can be as small as 2 and $\log (n+2) \sim \log n$. So all we have got is the generic main term for increment corresponding to the co-efficient $a_{i}$ will be $p_{i}\log a_{i}$. \\
For measuring the increase from $\log n$ to $\log n_{1}$ we try to estimate the increase for each $a_{i}$. Now \\
$[(p_{m}-1)\log a_{m} - \log p_{m}] \sim [(p_{m}-1)\log 2 - \log p_{m}] \sim [m\log m \log 2 - \log m - \log\log m] $ (using $p_{n} \sim n\log n$). \\
Now for the function \\
$ f(x) = x\log x \log 2 - \log x - \log\log x $ we seek to find its minimum and for that we solve for its derivative.\\
This clearly is the solution of the equation\\
$(\log 2) x(\log x)^{2} + (\log 2 x - 1)\log x = 1 $.\\
$\implies x = 0.130488$ or $2.39604$. \\ 

So from here we get that the minimum increase will be at-least\\
$(p_{m}-1)\log a_{m} - \log p_{m} \sim 2\log\log 2 - \log 2 - \log 2 \geq 0.634$.\\
So $a_{m}((p_{m}-1)\log a_{m} - \log p_{m}) \geq 2$ x  $0.634 = 1.268$
So evidently we have\\
$\log n_{1} - \log n \geq m. (1.26)$\\
$\implies \log_{10} n_{1} - \log_{10} n \geq 0.545$ $\nu(n)$\\where $\nu(n)$ is the number of distinct prime divisors of $n$. Since there are at least 2 distinct prime divisors with $a_{i} \geq 2$, we are done.

So by inductive argument we have the minimum size of $n$ for which $d^{k} (n) = 2$ occurs is at-least $10^{k}$.\\
Correspondingly $\forall$ $n$, $k$ has size $O(\log n)$.\\
\end{proof}

\begin{theorem}
For all sufficiently large integers $n$ its period has size $O(\log \log n)$ 
\end{theorem}
\begin{proof}
The bound for $k$ can be considerably improved for large $n$ using a well known result due to Wigart. See [4] for more information.\\

$\limsup\limits_{n}$  $\dfrac{\log d(n) \log\log n}{\log n} = \log 2$\\

which translates to: given $\epsilon > 0$, $\exists N_{0}$ such that $\forall n \geq N_{0}$ we have \\ 
\begin{eqnarray}
d(n) < n^{\frac{\log 2 (1 + \epsilon)}{\log\log n}}\\ 
\implies \log n > {\dfrac{\log\log n}{\log 2 (1 + \epsilon)}} \log d(n)
\end{eqnarray}

This clearly improves the bound on $k$. Assuming $d(n_{1}) = n$, we have to choose $n \geq max\big({N_{0},\dfrac{N_{1}}{10}}\big)$ where $N_{1}$ is the least integer such that $\log\log N_{1} \geq \log 2 (1+ \epsilon)(1+c)$\\

\begin{eqnarray}
\log n_{1} > {\dfrac{\log\log n_{1}}{\log 2 (1+\epsilon)}} \log d(n)\\
\implies \log n_{1} \geq (1+c) \log n
\end{eqnarray}
Here $c>0$ is a constant.\\
So we have by iteration $\log n_{1} \geq (1+c)^{k} \log 2$\\
So $k = O(\log\log n)$ for large enough n. \\
\end{proof}
We observe that :\\
k : 1   2   3    4  5    6 $\ldots$\\
n : 2   4   6  12   60  5040 $\ldots$\\  
Here given $k$ we have listed the least $n = n_{k}$ for which $d^{k}(n) = 2$. Now we make the following conjecture.\\

\begin{conjecture}
All the $n_{k}$'s which are produced by Theorem 1 are highly composite numbers. For a complete idea about what highly composite numbers are we refer to [1].\\
From a well known result(for more information about the source see [2]) we have:\\

$\max\limits_{n \leq x} d(n) = exp\bigg(\log2 \dfrac{\log x}{\log\log x} + O\big(\dfrac{\log x \log\log\log x}{(\log x)^{2}}\big)\bigg)$\\

So for large $n_{k}$ we expect that $\log n_{k-1} \sim \log 2 \dfrac{\log n_{k}}{\log\log n_{k}}$\\

$\max\limits_{n \leq n_{k}} d(n) = exp\bigg(\log2 \dfrac{\log n_{k}}{\log\log n_{k}} + O\big(\dfrac{\log n_{k} \log\log\log n_{k}}{(\log n_{k})^{2}}\big)\bigg)$\\

$\implies \max\limits_{n \leq n_{k}} d(n) \sim exp\big(\log 2 \dfrac{\log n_{k-1}}{\log 2}\big)$\\
$\implies \max\limits_{n \leq n_{k}} d(n) \sim n_{k-1}$
$\implies n_{k}$ is highly composite.
\end{conjecture}

\section{Acknowledgment}
We immensely thank Florentin Smarandache(University of New Mexico) to raise this problem in his paper ``THIRTY-SIX UNSOLVED PROBLEMS IN NUMBER THEORY". He also motivated us saying that nothing is known about the solution. We remain highly obliged to Prof. Balasubramanian(Institute of Mathematical Sciences) for guiding us to solve this problem up to whatever extent we have done.

\end{document}